\documentclass[10pt,reqno]{amsart}
\usepackage{amsthm,wrapfig,amsmath,amssymb,amsfonts,setspace,verbatim,graphicx,mathtools,mathrsfs,commath,float,mdframed,frame,xcolor,qtree,enumitem, ulem, afterpage}
\usepackage[hidelinks]{hyperref}
\usepackage[T1]{fontenc}
\usepackage[font=footnotesize,labelfont=sc]{caption}

\DeclareMathOperator{\eal}{Re}
\DeclareMathOperator{\imag}{Im}

\newtheorem{ut}{Theorem}

\newtheorem{upo}[ut]{Observation}
\newtheorem{up}[ut]{Proposition}
\newtheorem{lemma}[ut]{Lemma}

\newtheorem{uc}[ut]{Corollary}

\numberwithin{equation}{section}

\theoremstyle{definition}
\newtheorem*{ur}{Remark}

\theoremstyle{definition}

\begin{document}

\title[Dimension of radial Julia sets]{The topological dimension of radial Julia sets}

\subjclass[2010]{37F10, 30D05, 54F45} 
\keywords{radial Julia set, complex exponential, meandering set, zero-dimensional}
\address{Department of Mathematics, Auburn University at Montgomery, Montgomery 
AL 36117, United States of America}
\email{dsl0003@auburn.edu}
\author{David S.  Lipham}

\begin{abstract}  
We prove that the meandering  set for   $f_a(z)=e^z+a$ is homeomorphic to the space of irrational numbers  whenever $a$ belongs to the Fatou set of $f_a$. This extends recent results by Vasiliki Evdoridou  and  Lasse Rempe. It implies that  the radial Julia set of $f_a$ has topological dimension zero for all attracting and parabolic parameters, including all $a\in (-\infty,-1]$.  Similar results are obtained for Fatou's function  $f(z)=z+1+e^{-z}$.
\end{abstract}

\maketitle

\section{Introduction}

The focus of this paper is the class of exponential functions $$f_a(z)=e^z+a$$ for attracting and parabolic  parameters $a\in \mathbb C$.  The Fatou set $F(f_a)$ is the    set of normality for the family of iterates $\{f_a^n:n\in \mathbb N\}$, and the Julia set $J(f_a)=\mathbb C\setminus F(f_a)$ is its complement.   The   \textit{escaping set} $I(f_a)$ is the set of points $z\in \mathbb C$ such that  $f_a^n(z)\to\infty$, where   $\infty$ is the point at infinity on the Riemann sphere. It is known that $I(f_a)\subset J(f_a)$ in general, and   $f_a$ has an attracting or parabolic  cycle if and only if $a\in F(f_a)$ \cite[Proposition 2.1]{vas}.

The \textit{radial Julia set} $J_{\mathrm{r}}(f_a)$ is defined to be the set of points $z\in J(f_a)$ with the following property: There exists $\delta>0$ and infinitely many $n\in \mathbb N$ such that the spherical disc of radius $\delta$ around $f^n(z)$ can be pulled back univalently along the orbit of $z$.  For the parameters of interest in this paper, $J_{\mathrm{r}}(f_a)$ is approximately equal to the set of non-escaping points in $J(f_a)$. Indeed, $$J_{\mathrm{r}}(f_a)\subset J(f_a)\setminus I(f_a)$$ for all $a\in F(f_a)$, and  if $f_a$ has an attracting periodic orbit then $$J_{\mathrm{r}}(f_a)=J(f_a)\setminus I(f_a);$$ see \cite[Corollary 2.2]{vas}. The \textit{meandering set} $J_{\mathrm{m}}(f_a)$ is the larger collection of points  in $J(f_a)$ whose orbits  do not converge  to $\infty$ as quickly as possible. Precise definitions of $J_{\mathrm{m}}(f_a)$ and $J_{\mathrm{r}}(f_a)$ are provided in Section 2.

 All of the sets mentioned above will be viewed as topological subspaces of the complex plane $\mathbb C$.  A topological space $X$ is 
\begin{itemize}
\item \textit{totally separated} if for every two points $x,y\in X$  there is a clopen subset of $X$ which contains $x$ and misses $y$; and 
\item     \textit{zero-dimensional} if $X$ has a basis of clopen sets.
\end{itemize} 
Note that every zero-dimensional space is totally separated, but the converse is false. In fact, $J(f_a)$ contains totally separated spaces which are not zero-dimensional; see the Remark  in Section 3.2. 

Evdoridou and Rempe recently proved that if $a\in F(f_a)$,  then  $J_{\mathrm{m}}(f_a)\cup \{\infty\}$ (as a subspace of  the 
 Riemann sphere $\mathbb C\cup \{\infty\}$)    is  totally separated \cite[Theorem 1.3]{vas}. In this paper, we strengthen their result by proving  that $J_{\mathrm{m}}(f_a)$ and $J_{\mathrm{m}}(f_a)\cup \{\infty\}$ are  each homeomorphic to the zero-dimensional space of irrational numbers $\mathbb P$.   Since $J_{\mathrm{r}}(f_a)\subset J_{\mathrm{m}}(f_a)$ and zero-dimensionality is hereditary, we find that $J_{\mathrm{r}}(f_a)$ and $J_{\mathrm{r}}(f_a)\cup \{\infty\}$ are topologically zero-dimensional as well. 
 
  By contrast,  the Hausdorff dimension of $J_{\mathrm{r}}(f_a)$  is always greater than $1$; see \cite[Theorem 2.1]{uz} and  \cite[Theorem 2]{ka}. The topological dimension when $a\in J(f_a)$  can also differ. For instance, if $f_a$ is postsingularly finite (meaning that $\{f^n_a(a):n\in \mathbb N\}$ is finite),  then $J_{\mathrm{r}}(f_a)=J(f_a)\setminus I(f_a)$   contains an unbounded connected set; see \cite[Corollary 2.2(c)]{vas}, 
 \cite[Theorem 4.3]{sz1} and  \cite[Proposition 2.4(e)]{vas}. In this case $J_{\mathrm{r}}(f_a)$ and $J_{\mathrm{m}}(f_a)$ must have positive topological dimension.  A simple example of a postsingularly finite parameter is $a=\ln(\pi)+\frac{\pi}{2}i$. 
 
At the end of the paper we will apply our work to  Fatou's function $$z\mapsto z+1+e^{-z}.$$ This will  strengthen the  results of \cite{evd} and \cite[Section 5]{vas}. In particular, we will see  that the entire non-escaping set for Fatou's function is zero-dimensional.
\subsection*{Structure of the paper}
In Section 2 we will  state the  definitions of $J_{\mathrm{r}}(f_a)$,  $J_{\mathrm{m}}(f_a)$ and related sets, including the fast escaping set $A(f_a)$.  The  endpoints and connected components  of $J(f_a)$ are analyzed  in Section 3, and in Section 4 we  prove two crucial lemmas regarding $A(f_a)$. All of that will be integrated  in Section 5  to prove our main results.  And in Section 6  we consider Fatou's function.

  \section{Formal definitions}
  
  In the following definitions we assume that $f$ is a transcendental entire function, and $\widehat{\mathbb C}=\mathbb C\cup \{\infty\}$ is the Riemann sphere.

  \subsection{Definition of $J_{\mathrm{r}}(f)$}The radial Julia set was introduced by Urba\'{n}ski \cite{urb} and McMullen \cite{mcm}. It is defined as follows.
  
 For any $z_0\in \mathbb C$ let $\mathbb D_\delta(z_0)=\{z \in \widehat{\mathbb C} : \text{dist}(z_0, z) < \delta\}$, where the distance is with respect to the conformal metric on $\widehat{\mathbb C}$ (the unique metric of constant curvature $1$).   If $ A\subset \widehat{\mathbb C}$ and $z\in A$,  then let $\text{Comp}(z,A)$ denote the connected component of $z$ in $A$. Define $J_{\mathrm{r}}(f)$ to be the set of points $z\in J(f)$ with the  property: There exists $\delta > 0$ and infinitely many $n\in \mathbb N$   such that 
$$f^n: \text{Comp}(z,f^{-n}(\mathbb D_\delta(f^n(z))))\to \mathbb D_\delta(f^n(z))$$
 is a conformal isomorphism.  
 
  The definition of $J_{\mathrm{r}}(f)$ will not be  directly applied in this paper.  We will only need the fact $J_{\mathrm{r}}(f)\subset J(f)\setminus I(f)$ for the functions studied here. 

  \subsection{Definitions of $A(f)$ and $J_{\mathrm{m}}(f)$} The  \textit{fast escaping set} $A(f)$ was introduced by Bergweiler and Hinkkanen \cite{berg}, and  was studied further by Rippon and Stallard \cite{rip}. It is defined as follows.
  
   For each $r>0$ let  $M(r) =M(r,f)=  \max \{|f(z)|:|z|=r\}.$   Since $f$ is transcendental, we can choose $R>0$ sufficiently large so that $M^n(R) \to \infty$ and $M(r)>r$ for all $r\geq R$. Define  $$A(f) =\{z\in \mathbb C:\exists\;\ell\geq 0\text{ such that }f^{\ell+n}(z)\geq M^n(R)\text{ for all }n\in \mathbb N\}.$$
This definition is independent of the choice of $R$ by \cite[Theorem 2.2]{rip}.    Informally, $A(f)$ is the set of points $z\in \mathbb C$ such that $f^n(z)\to \infty$ as quickly as possible. Note that $A(f)\subset I(f)$. 

The set  $J_{\mathrm{m}}(f)$ is defined to be the complement $J(f)\setminus A(f).$ As previously indicated, $J_{\mathrm{m}}(f_a)\supset J_{\mathrm{r}}(f_a)$ for all $a\in F(f_a)$.
  

\section{Geometry of Julia sets}  
In this section we examine some important aspects of  $J(f_a)$ and its connected components. 
\textit{Throughout the section we  assume $a\in F(f_a)$.}

 \subsection{General structure of $J(f_a)$} For certain attracting parameters, including all $a\in (-\infty,-1)$,   $J(f_a)$ is homeomorphic to a  particular subspace $\overline X$ of $ \mathbb Z^\mathbb N \times [0,\infty)$. The set $\overline X$ is constructed in \cite{rem2} and is of the form 
$$\overline X=\bigcup_{\uline{s} \in \mathbb Z ^{\mathbb N}}  \{\uline{s}\}\times [t_{\uline{s}},\infty)$$ 
where each $t_{\uline{s}}\in[0,\infty]$ and $[\infty,\infty)=\varnothing$.   More generally,  by \cite[Theorem 9.1]{rem2}, for any parameter $a\in F(f_a)$  there is  a closed continuous  surjection $$H: \overline X\to J(f_a).$$   
If $H$ is a homeomorphism (e.g.\ when $a\in (-\infty,-1)$),  then $J(f_a)$ is a \textit{Cantor bouquet} in the sense of  \cite{aa} and \cite[Definition 2.7]{rem}. Otherwise,  $J(f_a)$  is a  \textit{pinched Cantor bouquet}.   Figures 1 and 2 show  a Cantor bouquet Julia set for $a=-2$, and  pinched Cantor bouquet Julia sets for three other parameters.  
   
We now indicate some features of the escaping set $I(f_a)$ and the pinching that may occur.  There is  a set $X\subset \overline X$ (the same $X$ defined in \cite{rem2}) such that:
   \begin{enumerate}
     \item $ \{\uline{s}\}\times (t_{\uline{s}},\infty)\subset X$ for all $\uline{s} \in \mathbb Z^{\mathbb N}$ (see \cite[Observation 3.1]{rem2}),
     \item  $H^{-1}(I(f_a))=X$ (see \cite[Corollary 5.3]{rem2}), and
     \item $H\restriction X:X\to I(f_a)$ is a homeomorphism (see \cite[Theorem 9.1]{rem2}).
    \end{enumerate}
Thus, if  $H^{-1}(z)$ is non-degenerate then  $z$ may be viewed as a point of $J(f_a)$ where multiple rays  of $\overline X$ are ``pinched together'' at their endpoints. Observe also that a non-degenerate pre-image  $H^{-1}(z)$ cannot intersect $X$, so the pinched points in $J(f_a)$ are all non-escaping (and thus belong to $J_{\mathrm{r}}(f_a)$ and $J_{\mathrm{m}}(f_a)$).

\subsection{Endpoints of $J(f_a)$} The endpoints of rays of $\overline X$ (the points $\langle \uline{s}, t_{\uline{s}}\rangle\in \overline X$)  are closely related to the endpoints of $J(f_a)$ that are  defined as follows.   

A point $z\in J(f_a)$ is on a \textit{hair}  if there exists an arc $\alpha:[-1,1]\hookrightarrow I(f_a)$ such that $\alpha(0)= z$. A point $z_0\in J(f_a)$ is an \textit{endpoint} if $z_0$ is not on a hair and there is an arc $\alpha:[0,1]\hookrightarrow J(f_a)$ such that $\alpha(0)=z_0$ and $\alpha(t)\in I(f_a)$ for all $t>0$. 

 The set of all endpoints of $J(f_a)$ is denoted $E(f_a)$.

\begin{upo}$E(f_a)=\{H(\langle \uline{s},t_{\uline{s}}\rangle):\uline{s}\in \mathbb Z^{\mathbb N} \text{ and } t_{\uline{s}}<\infty\}$.\end{upo}
 
 \begin{proof} If $t\in (t_{\uline{s}},\infty)$ then $H(\langle \uline{s}, t\rangle)$ is on a hair by items (1) and (3) above. Therefore $E(f_a)\subset\{H(\langle \uline{s},t_{\uline{s}}\rangle):\uline{s}\in \mathbb Z^{\mathbb N} \text{ and } t_{\uline{s}}<\infty\}$. We now prove the other inclusion. Let $z_0=H(\langle \uline{s},t_{\uline{s}}\rangle)$ be given. Since $\mathbb Z ^\mathbb N$ is totally separated,  every arc in $\overline{X}$ containing $\langle \uline{s},t_{\uline{s}}\rangle$ is contained in $ \{{\uline{s}}\}\times [t_{\uline{s}},\infty)$.  So by (2) and (3), every arc in $I(f_a)$ containing $z_0$ is contained in $H(\{{\uline{s}}\}\times [t_{\uline{s}},\infty))$. Since   $H\restriction \{{\uline{s}}\}\times [t_{\uline{s}},\infty)$ is a homeomorphism, it follows that $z_0$ is not on a hair.   On the other hand,  $\alpha(t)=H(\langle {\uline{s}},t_{\uline{s}}+t\rangle)$ defines an arc $\alpha:[0,1]\hookrightarrow J(f_a)$ with $\alpha(0)=z_0$ and $\alpha(t)\in I(f_a)$ for all $t>0$. Thus $z_0\in E(f_a)$. \end{proof}

\begin{ur}The endpoint set $E(f_a)$  is an example of a totally separated space which is  not zero-dimensional. This is a consequence of \cite[Theorem 1.7]{rem}, which states that $\infty$ is an explosion point for $E(f_a)$ (when $a\in F(f_a)$). This means that $E(f_a)$ is totally separated and $E(f_a)\cup \{\infty\}$ is connected.   The latter  implies that $E(f_a)$ is not zero-dimensional, since every clopen neighborhood in $E(f_a)$ must be unbounded.\end{ur}

\begin{figure}
\centering
\includegraphics[scale=.9]{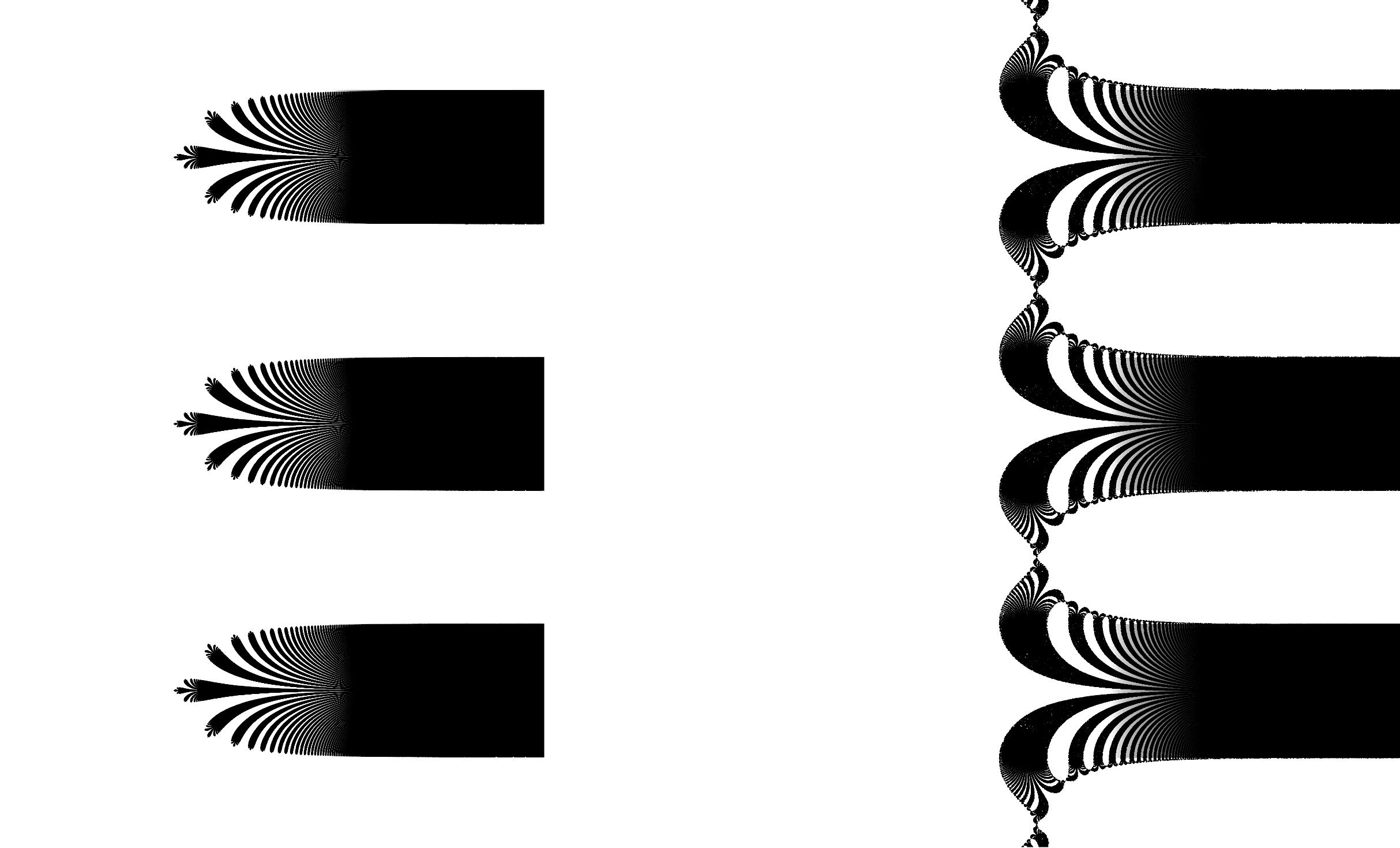}

\caption{Julia sets of $f_a$  for $a=-2$ (left) and $a=5+3.14 i$ (right).  }

\end{figure}

\begin{figure}

\centering
\includegraphics[scale=.9]{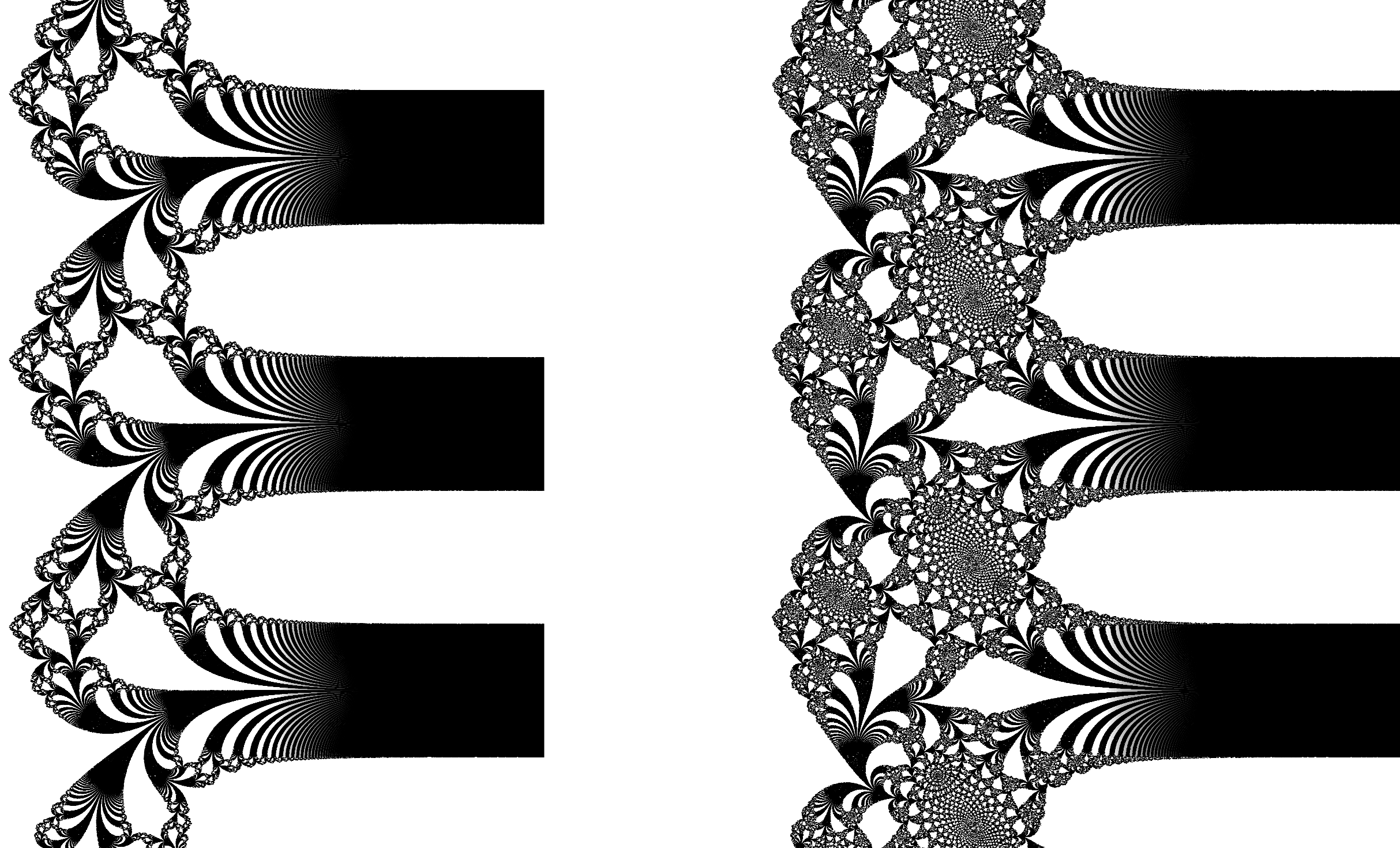}

\caption{Julia sets of $f_a$  for $a=2.06+1.57i$ (left) and $a=1.004+2.9i$ (right).  }
\end{figure}

   \subsection{Connected components of $J(f_a)$}  The following is a simple consequence of Observation 1  and the properties of $H$.
   
   \begin{up}\label{prop1}Each connected component of $J(f_a)$ contains a point of $E(f_a)$. \end{up}
   
\begin{proof}Let $C$ be a connected component of $J(f_a)$.  Since $H$ is surjective, there exists $\langle \uline{s},t \rangle\in H^{-1}(C)$. By continuity of $H$ the set  $H( \{{\uline{s}}\}\times [t_{\uline{s}},\infty))$ is connected, so we have   $H( \{{\uline{s}}\}\times [t_{\uline{s}},\infty))\subset C$.  Then $H(\langle \uline{s},t_{\uline{s}}\rangle)\in C\cap E(f_a)$ by Observation 1. 
\end{proof}
   
   The next proposition shows that every connected component of $J(f_a)$ can be written as an intersection of relatively clopen subsets of $J(f_a)$.
   
\begin{up}\label{prop2}For every connected component $C$ of $J(f_a)$ there is a sequence of open subsets $O_1\supset O_2\supset \ldots$ of $\mathbb C$ such that   $\partial O_n\subset F(f_a)$ and  $$\bigcap_{n\in \mathbb N}\overline{O_n}\cap J(f_a)=C.$$

\end{up}

\begin{proof} 
By  \cite[Proposition 6.15]{rem} there is a closed set $\gamma\subset F(f_a)$ such that every two endpoints $z_0,z_1\in E(f_a)$ have different itineraries with respect to $f_a^{-1}(\gamma)$. This means that there exists $n\geq 0$ such that $f_a^n(z_0)$ and $f_a^n(z_1)$ belong to different connected components of $\mathbb C\setminus f_a^{-1}(\gamma)$; see   \cite[Definition 6.1]{rem}. For such $n$, note that  $z_0$ and $z_1$ belong to different connected components of $\mathbb C\setminus f_a^{-1-n}(\gamma)$. 

Now let $C$ be a connected component of $J(f_a)$. Let $z_0\in C\cap E(f_a)$ be given by Proposition \ref{prop1}.   For each $n\in \mathbb N$ let  $O_n$ be  the connected component of $z_0$ in $\mathbb C\setminus \bigcup_{k=1}^n f_a^{-k}(\gamma)$. Clearly   $O_n$ is open, $O_{n}\supset O_{n+1}$, and $\partial O_n\subset  \bigcup_{k=1}^n f_a^{-k}(\gamma)\subset F(f_a)$ for every $n\in \mathbb N$. Also,  $\bigcap_{n\in \mathbb N} O_n\cap E(f_a)=\{z_0\}$ by the previous paragraph.   Since $O_n\cap  J(f_a)$ is clopen in $J(f_a)$, we have that  $\bigcap_{n\in \mathbb N}O_n\cap J(f_a)$ is a union of connected components of $J(f_a)$.  By Proposition \ref{prop1} we therefore have $\bigcap_{n\in \mathbb N}O_n\cap J(f_a)=C.$ Since $O_n\cap J(f_a)=\overline{O_n}\cap J(f_a)$, the proof is complete.
\end{proof}

The    proof above shows that every two elements of $E(f_a)$ belong to different connected components of $J(f_a)$.  Thus Proposition \ref{prop1} can be improved to:

   \begin{up}\label{prop3}Each connected component of $J(f_a)$ contains exactly one endpoint.
   \end{up}

The final proposition implies, in particular,   that each connected component of $J(f_a)$ is bounded in the imaginary direction.

\begin{up}\label{prop4}There exist $c,\delta>0$ and a closed set $\mathcal M\subset F(f_a)$ such that:
\renewcommand{\theenumi}{\roman{enumi}}
\begin{enumerate}
\item $\{z:\eal(z)\leq -c\}\subset \mathcal M\text{; and}$
\item $\sup \{|\imag (z)-\imag (z')|:z,z'\in O\}\leq \delta$  for every connected \\component $O$ of $\mathbb C\setminus \mathcal M. $
\end{enumerate}\end{up}

\begin{proof}The proof is contained in the first part of the proof of \cite[Theorem 3.1]{vas}. We include it here for completeness. Let $U$ be the component of $F(f_a)$ containing $a$. Let $c>0$ be large enough so that the closed disk $D:=\{z\in \mathbb C:|z-a|\leq e^{-c}\}$ is contained in $U$.  By \cite[Proposition 2.1]{vas} there is an arc $\sigma\subset U$ along which the real part tends to infinity, and which intersects $D$ at a single point.    Then $$\mathcal M:=f_a^{-1}(D\cup \sigma)\subset F(f_a)$$ is a closed set consisting of the left half-plane $\{z : \eal(z) \leq -c\}$ and  countably many arcs connecting this half-plane to $\infty$. The arcs are all $2\pi i \mathbb Z$-translates of each other, and each arc is bounded in the imaginary direction. Thus there exists $\varepsilon>0$ such that the imaginary part of each arc varies by less than $\varepsilon$. Each connected component $O$ of $\mathbb C\setminus \mathcal M$ is an open region between consecutive arcs that are $2\pi i$ apart, so $\delta=\varepsilon+2\pi$ is as desired. \end{proof}

\section{Key lemmas}

Assume $a\in F(f_a)$. Let $R>0$ be large enough so that $M^n(R)=M^n(R,f_a) \to \infty$ and $M(r)>r$ for all $r\geq R$. For example, based on the calculations in \cite[Lemma 2.5]{vas} we may take $R=3+2|a|$.

The first lemma (below) concerns the set  $$A_R(f_a):= \{z \in \mathbb C : |f_a^n(z)| \geq M^n(R)\text{ for all }n \geq0\},$$ and was mostly proved in \cite{vas}.

\begin{lemma}\label{lem1}There exists $\lambda>0$ such that for every $z_0\in J(f_a)$  there is an open set $V\subset \mathbb C$ such that $z_0\in V$,  $\partial V\cap J(f_a)\subset A_R(f_a)$, and $\sup\{\eal(z):z\in V\}\leq |z_0|+\lambda$.\end{lemma}

\begin{proof}Let $c,\delta>0$ and $\mathcal M\subset F(f_a)$ be given by Proposition \ref{prop4}.  Let $$\lambda=\max\{R,c,\ln(1+2(|a|+\delta)),  \ln(5+|a|)\}+6,$$ and let $z_0\in J(f_a)$ be given.  Let $R'=|z_0|+\lambda-3$.       To construct $V$ and verify its properties, we now follow the proof of \cite[Theorem 3.1]{vas} with $R'$ in the place of $R$. 

For each $k\geq 0$ let $X_k=\{z\in \mathbb C:|f^k_a(z)|\geq M^k(R')\}\cup \textstyle\bigcup_{j=0}^{k-1} f^{-j}_a(\mathcal M).$ Put $X=\bigcap_{k\geq 0} X_k.$ Observe that $X\subset A_{R'}(f_a)\cup F(f_a)$. We know $z_0\notin F(f_a)$, and  $z_0\notin A_{R'}(f_a)$  because $|z_0|<R'=M^0(R')$. Hence $z_0\notin X$. Let $V$ be the connected component of $z_0$ in $\mathbb C\setminus X$.  Since  $\partial V\subset X$ and $R'>R$, we get $$\partial V\cap J(f_a)\subset A_{R'}(f_a)\subset A_{R}(f_a).$$   

Using   (i) and (ii) from Proposition \ref{prop4}, and the fact $$R'>\max\{|z_0|,c,3,\ln(1+2(|a|+\delta))\},$$ the proof of \cite[Theorem 3.1]{vas} proceeds to show that   $\sup\{\eal(z):z\in V\}\leq K$ for a particular constant $K$  that depends on $R'$. Tracing the definition of $K$ back to  \cite[Corollary 2.7]{vas},  we find that $K=\max\{2+\ln(5+|a|),R'+3\}=R'+3=|z_0|+\lambda$.
  \end{proof}

We will also need the next lemma which compares the orbits of fast escaping points and meandering points.

\begin{lemma}\label{lem2} Let $s\in A(f_a)$ and $z_0\in \mathbb C\setminus A(f_a)$.  Then  for every $\kappa>0$ there exists $n\in \mathbb N$ such that  $|f_a^{n}(s)|>2|f^n_a(z_0)|+\kappa$ and $\eal(f^{n}_a(s))>0$.  
\end{lemma}

\begin{proof}Let $\kappa>0$.  Let $\ell\geq 0$ be such that $f_a^{\ell+n}(s)\geq M^n(R)$ for all $n\in \mathbb N$.

By  \cite[Lemma 2.5]{vas},   $M(r)-r\geq e^{r-1}-r> r+\kappa$ when $r$ is sufficiently large.     Since $M^n(R)\to\infty$, we may substitute   $r=M^n(R)$ to see that  there exists $m\in \mathbb N$ such that  $M^{n+1}(R)-2M^n(R)> \kappa$ for all $n\geq m$.  By increasing $m$, we may assume $|f_a^{n}(s)|>|a|+1$ for all $n\geq m$.   
Since $z_0\notin A(f_a)$, there exists $N\in \mathbb N$ such that $$|f^{\ell+m+1+N}_a(z_0)|< M^{N}(R).$$ Let $n=\ell+m+N+1$. Then $$2 |f_a^{n}(z_0)|<  2M^{N}(R)\leq 2M^{m+N}(R)< M^{m+N+1}(R)-\kappa\leq  |f_a^{n}(s)|-\kappa.$$ Therefore $|f_a^{n}(s)|>2|f^{n}_a(z_0)|+\kappa$.  Also,  $|f_a^{n+1}(s)|>|a|+1$ implies $\eal(f^{n}_a(s))>0$ since $f_a$ maps each point with negative real part into the unit disc around $a$.\end{proof}

\section{Main results}

We will now  combine Propositions \ref{prop2} through \ref{prop4} with Lemmas 6 and 7 to compute the topological dimension of $J_{\mathrm{m}}(f_a)$.

\begin{ut} \label{main}If $a\in F(f_a)$ then $J_{\mathrm{m}}(f_a)$ is zero-dimensional. \end{ut}

\begin{proof}Let $z_0\in J_{\mathrm{m}}(f_a)$, and let $U$ be any bounded open subset of $\mathbb C$ with $z_0\in U$.  We will construct a relatively clopen subset of $J_{\mathrm{m}}(f_a)$ which contains $z_0$  and is contained in $U$. To that end, let   $C$ be the connected component of $z_0$ in $J(f_a)$.  Let $ S=C\cap \partial U$. By Proposition \ref{prop3} and \cite[Proposition 2.4(b),(c)]{vas} we have $$S\subset J(f_a)\setminus E(f_a)\subset A(f_a).$$

Let $\delta$ and $\mathcal M$ be given by Proposition \ref{prop4}. Let $\lambda>0$ be given by Lemma \ref{lem1}, and for each $n\in \mathbb N$ define $$S_n=\{s\in S:|f_a^n(s)|>2|f_a^n(z_0)|+ \lambda+\delta \text{ and }\eal(f_a^n(s))> 0 \}.$$ By Lemma \ref{lem2} and compactness of $S$,  there is a finite $\mathcal F\subset \mathbb N$ such that $S\subset \bigcup \{S_{n}:n\in \mathcal F\}$.   For each $n\in \mathcal F$, Lemma \ref{lem1} provides an open set    $V_n\subset \mathbb C$ such that:   
\begin{itemize}\renewcommand{\labelitemi}{\scalebox{.5}{{$\blacksquare$}}}
\item $f^{n}_a(z_0)\in V_n$, 
\item $\partial V_n\cap J(f_a)\subset A_{R}(f_a)$,  and 
\item  $\sup\{\eal(z):z\in V_n\} \leq  |f^{n}_a(z_0)|+\lambda.$   
\end{itemize}
Notice that we can shrink $V_n$ to obtain $$\sup \{|\imag (z)-\imag (f_a^{n}(z_0))|:z\in V_n\}\leq \delta$$ in addition to the three properties above; simply intersect $V_n$ with the connected component of $f^n_a(z_0)$ in $\mathbb C\setminus \mathcal M$. Now if $z\in V_n$ and $\eal(z)> 0$, then $$|z|\leq \eal(z)+|\imag(z)|\leq 2|f_a^{n}(z_0)|+\lambda+\delta.$$   Therefore    $f^{n}_a(S_{n})\cap  \overline {V_n}=\varnothing$. So $S_{n}\cap  \overline {f^{-n}_a(V_n)}=\varnothing.$

 Let $V=\bigcap \{f_a^{-n}(V_n):n\in \mathcal F\}$.  Then $V$ is open, and   \begin{equation} S\cap \overline V=\varnothing. \end{equation}
Further, $\partial f_a^{-n}(V_n)\cap J(f_a)\subset  f_a^{-n}[A_R(f_a)]\subset A(f_a)$ implies  \begin{equation}\partial V\cap J(f_a)\subset A(f_a).\end{equation}By Proposition \ref{prop2}, there is a collection of open sets   $\{O_n:n\in \mathbb N\}$ such that    \begin{equation}\partial O_n\subset F(f_a), \end{equation}
\begin{equation}\overline{O_{n}}\supset \overline{O_{n+1}}, \text{ and}\end{equation} 
\begin{equation}\bigcap_{n\in \mathbb N }\overline{O_n}\cap J(f_a)=C. \end{equation} 
By (5.1),  (5.4), (5.5),  and compactness of $\overline {O_n }\cap  \overline{V}\cap \partial U\cap  J(f_a)$, there exists $n\in \mathbb N$ such that 
  \begin{equation} \overline {O_{n} }\cap \overline{V}\cap \partial U\cap J(f_a)=\varnothing.\end{equation} Then $O_{n}\cap V\cap  U\cap  J_{\mathrm{m}}(f_a)$ is a relatively clopen subset of $J_{\mathrm{m}}(f_a)$ (it is closed in $J_{\mathrm{m}}(f_a)$ by  (5.2),   (5.3), and (5.6)). It contains $z_0$ and is contained in $U$.
  \end{proof}

Now the  topological type of $J_{\mathrm{m}}(f_a)$ can  be determined using the Alexandroff-Urysohn characterization of the irrationals. The Alexandroff-Urysohn characterization \cite[Theorem 1.9.8]{van} states that every topologically complete, nowhere compact, zero-dimensional space is homeomorphic to $\mathbb P$.    

A space $X$ is  \textit{topologically complete} if there exists a metric $d$ on $X$ such that $(X,d)$ is a complete metric space that is homeomorphic to $X$ in its original topology. $G_{\delta}$-subsets of $\mathbb C$ are topologically complete by \cite[Theorem A.6.3]{van}.  

A space $X$ is \textit{nowhere compact} if every compact subset of $X$ has empty interior in $X$.  We observe that if $X\subset Y\subset \mathbb C$,  $X$ is dense in $Y$, and $Y\setminus X$ is dense in  $Y$, then $X$ is nowhere compact.  For let $K$ be any compact subset of $X$ and let $x\in K$.  If $Y\setminus X$ is dense in $Y$, then there is a sequence of points $y_1,y_2,\ldots\in Y\setminus K$ converging to $x$.  By compactness of $K$ we have that  $Y\setminus K$ is open in $Y$. Thus, assuming $X$ is dense in $Y$, there is a sequence of points $x_1,x_2,\ldots\in  X\setminus K$ with $d(x_n,y_n)<\frac{1}{n}$. Then $x_n\to x$, proving that $K$ is not a neighborhood of $x$ in $X$.  Since $x$ was an arbitrary point of $K$, we conclude that $K$ has empty interior in $X$.

\begin{uc}If $a\in F(f_a)$ then $J_{\mathrm{m}}(f_a)\simeq J_{\mathrm{m}}(f_a) \cup \{\infty\}\simeq \mathbb P$.  \end{uc}

\begin{proof}The Julia set $J(f_a)$ is closed in $\mathbb C$, and    $$A(f_a) =\bigcup_{\ell\in \mathbb N} f_a^{-\ell}[A_R(f_a)]$$ is $F_{\sigma}$ in $\mathbb C$.  So $J_{\mathrm{m}}(f_a)$ is a $G_{\delta}$-subset of $\mathbb C$ and is therefore   topologically complete. The sets $J_{\mathrm{m}}(f_a)$ and $J(f_a)\setminus J_{\mathrm{m}}(f_a)(=A(f_a))$ are infinite and completely invariant under $f_a$, so by  Montel's theorem they are both dense in $J(f_a)$. By the observation above (with $Y=J(f_a)$ and $X=J_{\mathrm{m}}(f_a)$),  $J_{\mathrm{m}}(f_a)$ is nowhere compact.  Finally,  $J_{\mathrm{m}}(f_a)$ is zero-dimensional by Theorem \ref{main}.   By  
    \cite[Theorem 1.9.8]{van} we have $J_{\mathrm{m}}(f_a)\simeq \mathbb P$. The one-point extension $J_{\mathrm{m}}(f_a)\cup \{\infty\}$ is also zero-dimensional by Theorem 8 and \cite[Theorem 3.11]{dim}.   Hence the characterization \cite[Theorem 1.9.8]{van} can be applied again to see that $J_{\mathrm{m}}(f_a)\cup \{\infty\}\simeq \mathbb P$.  \end{proof}

\begin{uc}If $a\in F(f_a)$ then $J_{\mathrm{r}}(f_a)$  and $J_{\mathrm{r}}(f_a)\cup\{\infty\}$ are zero-dimensional.\label{rad} \end{uc}

\begin{proof}This follows immediately from Corollary 9 and the fact $J_{\mathrm{r}}(f_a)\subset J_{\mathrm{m}}(f_a)$.\end{proof}

\section{Fatou's function}
The function  defined by  $f(z)=z+1+e^{-z}$ is called  \textit{Fatou's function} \cite{fat}.  
Here it is known that $J_{\mathrm{r}}(f)=\mathbb C\setminus  I(f)=J(f)\setminus  I(f)\subset J_{\mathrm{m}}(f)$ [citation to appear].

Previous work has shown that $J_{\mathrm{r}}(f)\cup\{\infty\}$  and $J_{\mathrm{m}}(f)\cup\{\infty\}$ are totally separated; see  \cite[Theorem 5.2]{evd} and \cite[Theorem 5.1]{vas}.   The corollaries below, together with  \cite[Theorem 3.11]{dim} (zero-dimensionality and one-point extensions), imply that both of these sets are in fact zero-dimensional.

\begin{uc}$J_{\mathrm{m}}(f)$ is homeomorphic to $\mathbb  P$.\end{uc}

\begin{proof}Define $h:\mathbb C\to \mathbb C$ by $h(z)=e^{-1}ze^{-z}$. The proof of \cite[Theorem 5.1]{vas} shows that  $J_{\mathrm{m}} (f )$ is covered by a discrete collection of  homeomorphic copies of $J_{\mathrm{m}}(h)$.  And  $J_{\mathrm{m}}(h)\simeq J_{\mathrm{m}}(f_{-2})$ by \cite[Proposition 5.2]{vas}. We can now apply Theorem 8 with $a=-2$ to see that $J_{\mathrm{m}}(f)$  is zero-dimensional. Further, the proof of Corollary 9 shows that $J_{\mathrm{m}}(f)\simeq \mathbb P$.  \end{proof}

\begin{uc}$J_{\mathrm{r}}(f)$ is zero-dimensional.\end{uc}

\begin{proof}This follows from Corollary 11 and the fact $J_{\mathrm{r}}(f)\subset J_{\mathrm{m}}(f)$.
\end{proof}

\subsection*{Acknowledgements}We thank the referee for their careful reading and suggestions which greatly improved the paper.

\end{document}